\documentclass[a4paper, 10pt, reqno, draft]{amsart}
\usepackage[utf8x]{inputenc}
\usepackage[toc]{appendix}
\usepackage{amssymb,amscd,mathtools ,amsmath, amsthm, commath, amsfonts}
\usepackage{comment, fullpage, anysize}
\usepackage{color}
\marginsize{1.5in}{1.5in}{1in}{1in}

\newtheorem{theorem}{Theorem}[section]
\newtheorem{proposition}{Proposition}[section]
\newtheorem{corollary}{Corollary}[section]
\newtheorem{mydef}{Definition}
\newtheorem{lemma}[theorem]{Lemma}
\theoremstyle{remark}

\newcommand{\eps}{\varepsilon}
\newcommand{\R}{\mathbb{R}}
\newcommand{\N}{\mathbb{N}}
\newcommand{\Z}{\mathbb{Z}}
\newcommand{\X}{\mathcal{X}}
\newcommand{\T}{\mathbb{T}}
\newcommand{\MExp}{\mathbb{E}}
\newcommand{\pP}{\mathbb{P}}
\newcommand{\sS}{\mathbb{S}}

\newcommand{\m}{\boldsymbol{m}}
\newcommand{\tm}{\tilde{\boldsymbol{m}}}
\newcommand{\mn}{\boldsymbol{m}_n}
\newcommand{\tmn}{\tilde{\boldsymbol{m}}_n}
\newcommand{\h}{\boldsymbol{h}}
\newcommand{\vv}{\boldsymbol{v}}
\newcommand{\xxi}{\boldsymbol{\xi}}
\newcommand{\w}{\boldsymbol{w}}
\newcommand{\uu}{\boldsymbol{u}}
\newcommand{\spn}{\textup{span}}
\newcommand{\e}{\hat{\boldsymbol{e}}}
\newcommand{\Heff}{\boldsymbol{H}_{\text{eff}}}

\DeclareMathOperator*{\esssup}{ess\,sup}

\numberwithin{equation}{section}

\allowdisplaybreaks

\title{Strong solvability of regularized stochastic Landau-Lifshitz-Gilbert equation}
\author{Olga Chugreeva}
\address{Lehrstuhl I f\"{u}r Mathematik, RWTH Aachen University,
Pontdriesch 14-16, 52056 Aachen, Germany}
\email{olga@math1.rwth-aachen.de}
\author{Christof Melcher}
\address{Lehrstuhl I f\"{u}r Mathematik  \& JARA Fundamentals of Future Information Technologies, RWTH Aachen University,
Pontdriesch 14-16, 52056 Aachen, Germany}

\begin{document}
\begin{abstract}   
We examine a stochastic Landau-Lifshitz-Gilbert equation based on an exchange energy functional containing second-order derivatives of the unknown field. Such regularizations are featured in advanced micromagnetic models recently introduced in connection with nanoscale topological solitons. We show that, in contrast to the classical stochastic Landau-Lifshitz-Gilbert equation based on the Dirichlet energy alone,
the regularized equation is solvable in the stochastically strong sense. As a consequence it preserves the topology of the initial data, almost surely. 
\end{abstract}
\maketitle
\section{Introduction}
The Landau-Lifshitz-Gilbert equation is the fundamental evolution law in the theory of ferromagnetism, first introduced by Landau and Lifshitz \cite{Landau_Lifshitz} and later reformulated by Gilbert \cite{Gilbert}. The equation describes the dynamics of a direction field $\m$ on a spatial domain $D\subset \R^{3}$, i.e., 
\[\m:\,D\to\sS^{2},
\]
representing a local magnetization distribution. In the form initially studied by Landau and Lifshitz, the nondimensionalized equation reads 
\begin{equation}\label{LLG_general}
\partial_{t}\m=-\lambda \m\times\m\times \Heff(\m)-\m\times \Heff(\m),
\end{equation}
where $\lambda>0$ is a damping parameter and $\Heff(\m)$ is the \textit{effective field}, which defines an operator $\Heff$ on the space of admissible magnetization fields. The evolution described by~\eqref{LLG_general} is mathematically a hybrid of conservative and dissipative dynamics. The first term on the right-hand side is the dissipative damping term, the second is the conservative precession term. 

Stochastic versions of~\eqref{LLG_general} have been considered in physics literature from 1960s on, starting with the work of Brown \cite{brown1963thermal}. Regarding the way of introducing the noise into equation~\eqref{LLG_general}, two ideas are widely accepted (cf.~\cite{Bertotti2009nonlinear}, Chapter~9). First, the noise $\xxi$ should be a part of the effective field, i.e., $\Heff(\m)$ in~\eqref{LLG_general} is replaced by~$\Heff(\m)+\xxi$. Second, it is sufficient to consider the noisy precession \cite{kubo1970brownian}. This leads to the stochastic equation 
\begin{equation}\label{SLLG_general}
d\m= - (\lambda \m\times\m\times \Heff(\m)+\m\times \Heff(\m))\,dt - \m\times\xxi,
\end{equation} 
with determinisitic and stochastic ingredients $\Heff(\m)$ and $\xxi$, respectively.
 
The effective field $\Heff(\m)$ is the negative functional gradient of the governing energy~$E(\m)$. Therefore, the concrete form of~\eqref{LLG_general} and~\eqref{SLLG_general} crucially depends on the choice of $E(\m)$. In the full micromagnetic model, the governing energy functional includes exchange, anisotropy, Zeeman, and magnetostatic energy, each encoding an interaction of a particular kind~\cite{hubert2008magnetic, Visintin_LLG}. The analysis of the complete energy functional is rather elaborate, especially because of the nonlocal character of the magnetostatic energy. Regularity and well-posedness issues, however, are mostly associated to small scale properties governed by the exchange energy as highest order energy contribution. Exchange interaction is usually modelled by the Dirichlet energy of $\m$, 
\begin{equation}\label{eq:ExchDir}
E_{\text{exch}}(\m)=\frac{1}{2}\int\limits_{D}|\nabla \m|^{2}\, dx,
\end{equation} 
which penalizes rapid changes in the magnetization direction, and is minimized by a uniform magnetization, making the ferromagnetic ground state. 

If the energy is given by~\eqref{eq:ExchDir}, the effective field is~$\Heff(\m)=\Delta \m$. 
The associated version of \eqref{SLLG_general} is a quasilinear stochastic partial differential equation, i.e. a parabolic system to be more precise, raising the question of appropriate notions of solvability. Recently, Brzezniak, Goldys and Jegarai \cite{Brzezniak} (see also~\cite{BrzezniakLi}) have constructed a solution for \eqref{SLLG_general} with this canonical choice of the effective field, i.e.,
\begin{equation}\label{SLLG_Br}
 d\m=-(\lambda\,\m\times\m\times\Delta\m+\m\times\Delta\m)\,dt+(\m\times \h)\circ dB_t,
\end{equation}
with a sufficiently regular vector field~$\h$ and a one-dimensional Brownian motion~$B_{t}$ interpreted in the Stratonovich sense in order to preserve the constraint $|\m|=1$. Analytically, their solution is a weak solution with improved regularity. In particular, the obtained a-priori estimates on the precession term $\m\times\Delta\m$ guarantee that $\m$ solves~\eqref{SLLG_Br} in the $L^{2}(D)$-sense, almost surely. However, from the probabilistic point of view the solution is weak in the sense that the solution and the driving Brownian motion are constructed simultaneously on \textit{some} probability space and are thus interconnected. The approach, however, is also appropriate 
to design numerical schemes, and indeed stimulated activity in numerical analysis~\cite{alouges2014semi, banas2013computational, banas2014convergent, dunst2015convergence, goldys2016finite}.

We are interested in the possibility of solving~\eqref{SLLG_Br} in the stochastically strong sense, aiming to construct a solution for \textit{any} probability space and \textit{any} Brownian motion prescribed in advance. 

Our motivation is twofold. First, the physically reasonable interpretation, which sees the noise as an additional random input and the solution as the immediate response to it, is justified only for a stochastically strong solution. In more abstract terms, a stochastically strong solution is an image of the driving Brownian motion under certain mapping (\cite{Karatzas_Shreve}, Chapter~5.1). In this case, we can indeed establish a relation of causality between the solution and the driving Brownian motion. For a stochastically weak solution, such relation is not necessary at hand. Second, the stochastic sense of solvability determines which properties of the solution can in principle be studied later. A good example of such study is the work by Kohn et al.~\cite{kohn2005magnetic} on the stochastic Landau-Lifshitz-Gilbert equation with spatially uniform magnetization. Large deviations theory, which is the main tool in this work, requires unique solvability in the stochastically strong sense.

More generally, the notion of solvability for stochastic (partial) differential equations is closely related to the uniqueness properties of the solution. This is the essence of the result of Yamada and Watanabe~\cite{YamadaWatanabe} and its generalizations~\cite{Gyongy_Krylov, OndrejatDiss}. Informally speaking, if a stochastically weak solution of an S(P)DE is unique in the strong, so-called pathwise, sense, then the equation is actually solvable in the stochastically strong sense. This suggests the general strategy of proving the existence and uniqueness of a stochastically strong solution by first constructing a stochastically weak solution and then showing that the solution is pathwise unique. In dimension one, this approach has  successfully been applied to~\eqref{SLLG_Br} in~\cite{BrzezniakLDForLLG}. 

In contrast, this approach does not seem to be directly applicable to the stochastic Landau-Lifshitz-Gilbert equation~\eqref{SLLG_Br} in dimension two and higher. For instance, it is known that the weak solutions to the corresponding deterministic equation are not unique~\cite{Alouges_Soyeur}, unless the solution is energy decreasing in dimension two~\cite{Harpes_Uniqueness_2004}, or the initial data satisfies a smallness condition in dimension three or higher~\cite{Melcher_LLG}. 

In this work, we propose a regularized version of \eqref{SLLG_Br} that we can treat in dimensions two and three according to the strategy described above. In what follows, we discuss only the three-dimensional case, the arguments for the two-dimensional case are identical. The regularization stems from an advanced exchange energy, that still exhibits $\mathrm{O}(3)$ invariance in real- and magnetization space. Instead of ~\eqref{eq:ExchDir}, we consider the energy functional 
\begin{equation}\label{eq:RegularizedLLEnergy}
E_{\eps}(\m)=\frac{1}{2}\int\limits_{D}\eps^{2}|\nabla^{2}\m|^{2}+|\nabla \m|^{2}\, dx,
\end{equation}
where a new positive parameter $\eps$ is fixed. Mathematically speaking, we model the governing energy functional by the norm of 
$\m$ in the Sobolev space~$H^{2}$ instead of~$H^{1}\!$, controlling the modulus of continuity. In the context of an advanced micromagnetic model introduced in \cite{Rybakov_et_al_2017}, second order terms, including the one used in our regularization, arise from the classical discrete Heisenberg model in a continuum limit beyond nearest neighbor interactions (see also \cite{bogolubsky1988three}). Second order terms where also proposed to stabilize the skyrmion solutions of the Landau-Lifshitz-Gilbert equation~\cite{abanovPokrovsky, kirakosyanPokrovsky} in dimension two. 

The regularized Landau-Lifshitz-Gilbert equation corresponding to~$E_{\eps}(\m)$ is
\begin{equation}\label{LLG}
 d\m=(\lambda\,\m\times\m\times(\eps^{2}\Delta^{2}\m-\Delta\m)+\m\times(\eps^{2}\Delta^{2}\m-\Delta\m))\,dt+(\m\times \h)\circ dB_t.
\end{equation}
Equation~\eqref{LLG} retains the general structure prescribed by~\eqref{SLLG_general}. Together with the Stratonovich formulation, this ensures that the geometric properties of the solution are preserved. As we shall see later, the solution of~\eqref{LLG} stays on the unit sphere. 

We work on the torus $\T^{3}$ for the sake of simplicity. This setting is intended as a reasonable model of a small region in the bulk of a ferromagnet away from the boundary. Our results are valid for the noise of a more general form 
$
\sum_{k=1}^{N} \h_k(x)\circ dB_t^k,
$
with smooth vector fields~$\h_k$ and independent one-dimensional Brownian motions~$B_t^k$, but we consider a single Brownian motion to keep our presentation transparent. 

In Theorem~\ref{Thm:strong} below, we shall prove that~\eqref{LLG} has a unique stochastically strong solution. Like the classical stochastic Landau-Lifshitz-Gilbert equation~\eqref{SLLG_Br}, equation~\eqref{LLG} holds in the $L^{2}$-sense. Moreover, by interpolation (Corollary~\ref{Cor:Interpolation}), the solution of~\eqref{LLG} is almost surely continuous in space and time. It therefore preserves (almost surely) the topology of the initial data, analytically described in terms of the so-called Hopf-Pontryagin invariants, see \cite{AucklyKapitanski} for a modern analytical approach and Section~3 in~\cite{GaussMaps} for a modern geometric exposition. In the context of string-like topological solitons in magnets as examined in \cite{Rybakov_et_al_2017}, the situation of fields $\m:\R^3 \to \mathbb{S}^2$ defined on the entire Euclidean space is particularly relevant. In this case the homotopy type is described by a single integer, the Hopf invariant that corresponds to the geometric linking number of two generic fibres. In view of well-established existence results in the deterministic case~\cite{Alouges_Soyeur, MelcherPtashnyk} it is conceivable that one can construct a solution of~\eqref{LLG} with analogous properties on the whole~$\R^{3}$. 

Our arguments rely strongly on the work of Brzezniak et al.\,\cite{Brzezniak} and on the classical result of Yamada and Watanabe~\cite{YamadaWatanabe}, more precisely, on the corresponding result for weakly convergent sequences due to Gy\"ongy and Krylov~\cite{Gyongy_Krylov} (cf.\,Lemma~\ref{Lem:Gyongy_Krylov}). We first construct, by means of the Galerkin approximation, a stochastically weak (martingale) solution of~\eqref{LLG}. Since the a-priori estimates imply only weak convergence of the approximating sequence, we require the Skorokhod representation theorem to identify the limit. Consequently, we have to switch to some canonical probability space, and the solution obtained in this manner is stochastically weak. In the second step, we show that the approximating sequence converges almost surely on the original probability space, which implies that its almost sure limit is a stochastically strong solution. By Lemma~\ref{Lem:Gyongy_Krylov}, this amounts to showing that the martingale solution obtained in the first step is pathwise unique, i.e., unique in the stochastically strong sense. The proof of uniqueness is direct: We derive an estimate for the $L^{2}(\Omega;L^{\infty}(0,T;L^{2}(\T^{3})))$-norm of the difference of two solutions. It turns out that the equation for the difference is deterministic. The main difficulties are thus of analytical rather than probabilistic nature. The new equation 
includes nonlinear terms containing the second derivatives of the solutions.  
To obtain the required estimates, we use classical interpolation and product inequalities for functions belonging to certain Sobolev spaces. 

Our proof of uniqueness exploits specific form of the energy functional~\eqref{eq:RegularizedLLEnergy} in a crucial way, for two reasons. First, for two hypothetical solutions~$\m_{1}$ and~$\m_{2}$, we need a good control over the quantity~$\lVert \m_{1}-\m_{2}\rVert_{L^{\infty}}$ appearing in certain nonlinear estimates. In the present case, it can be interpolated by the $H^{2}$ and $L^2$-norm of the increment. This follows essentially from the Sobolev embedding $H^{2}\hookrightarrow L^{\infty}$, valid up to dimension three. For equation~\eqref{SLLG_Br}, such control is not available due to the failure of the Sobolev embedding of~$H^{1}$ into~$L^{\infty}$ in dimension two and higher.

Second, our uniqueness argument exploits the specific algebraic structure of the quasilinear terms arising from precession and geometry (cf. Proposition~\ref{Prop:IdentitiesForM}). Some key estimates therefore rely strongly on classical differential calculus rather than pseudo-differential calculus for fractional derivatives. Therefore the arguments do not extend to regularizations of the form
\[
E(\m)=\frac{1}{2}\int\limits_{\T^{3}}\eps^{2}|\nabla^{s}\m|^{2}+|\nabla\m|^{2}\,dx
\]
suggested by the Sobolev embedding of $H^{s}$ into~$L^{\infty}$ valid for $s>3/2$ in three space dimensions. For the purpose of constructing (weak) martingale solutions, however,
a fractional regularizations can be treated in the framework of~\cite{Brzezniak}.

In the context of string-like topological solitons, it would be interesting to extend our construction to the whole~$\R^{3}$.  In this case an approximation scheme based on spatial discretization \cite{Alouges_Soyeur} is the method of choice, though the argument seems to be technically more demanding than that for the Galerkin approximation we use on the torus. Indeed, one has to show that the approximating equation, which is an infinite-dimensional stochastic ODE with a polynomial drift, has a global in time solution. In such a situation, the standard approach~\cite{brzezniak1997stochastic} consists of first obtaining the existence of a solution up to a blow-up time and then showing, with the help of a-priori estimates on the energy, that the blow-up time is almost surely infinite.

The work is organized as follows. In Section~2, we discuss technical details concerning equation \eqref{LLG}, give the necessary definitions and formulate our results. In Section~3, we outline the main steps in the construction of a weak martingale solution of~\eqref{LLG}. Finally, we prove the pathwise uniqueness of the solution in Section~4. 

\textbf{Acknowledgements.} 
The first author acknowledges support from the German Academic Exchange Service (DAAD) Doctoral Training Grant (A/10/86352). 
The second author acknowledges support from JARA FIT seed funds.

\section{Preliminaries}
\subsection{General notation}\label{Sec:Notation}

We denote by $L^{p}$ and $H^{k}=W^{2,k}$ the spaces of Lebesgue and Sobolev functions or fields on $\T^{3}=\R^{3}/\Z^{3}$, respectively.
We will often use the following facts about the space $H^{2}$. 
\begin{lemma}\label{prop:Sobolev}
\begin{enumerate}
For the Sobolev space~$H^{2}$ in dimension three, the following holds. 
\item The space~$H^{2}$\! is continuously embedded into $L^{\infty}$\!. There exists a constant $C$ such that for all $u$ from~$H^{2}$, there holds
\[\lVert u\rVert_{L^{\infty}}+ \| \nabla u\|_{L^6} \leqslant C\lVert u\rVert_{H^{2}}.
\]
\item There exists a constant $C$ such that for all $u$ and $v$ from $H^{2}$, there holds
\begin{equation}\label{Moser1}
\lVert uv\rVert_{H^{2}}\leqslant C(\lVert u\rVert_{H^{2}} \lVert v\rVert_{L^{\infty}}+\lVert v\rVert_{H^{2}}\lVert u\rVert_{L^{\infty}}).
\end{equation}
In particular, the space $H^{2}$ is closed under pointwise multiplication, with the estimate 
\[\lVert uv\rVert_{H^{2}}\leqslant C(\lVert u\rVert_{H^{2}}^{2}+\lVert v\rVert_{H^{2}}^{2}).
\]
\end{enumerate}
\end{lemma}
\begin{proof}
The first estimates are a particular case of the Sobolev embedding theorem~\cite{Evans}, Theorem~5.6. The product estimate follows directly from claim (1), cf.~\cite{Taylor_III}, Proposition~13.3.7.
\end{proof}

For a function $f$ from $H^{k}$, the homogeneous seminorm $\|f\|_{\dot{H}^{k}}$ is the $L^{2}$-norm of the highest ($k$th) derivative of $f$:
\[\|f\|_{\dot{H}^{k}}:=\|D^{k}f\|_{{L}^{2}}.
\]
For positive $\beta$ it is customary to define the fractional Sobolev space $H^\beta$ by means of the Fourier transform. The space $H^{-\beta}$ is the dual of $H^{\beta}$.

For a Banach space $X$ and a measure space $(E,\mu)$, the space $L^{p}(E;X)$ with $ p\in[1, \infty)$ is the Bochner space of $\mu$-measurable functions $f : E\rightarrow X$ such that $\int_{E}\lVert f\rVert_{X}^{p}d\mu<\infty$. In this work, $(E,\mu)$ is either an interval $[0, T]$ with the Lebesgue measure or a probability space $(\Omega, \pP)$.

For a separable Banach space $X$ and parameters $\alpha\in(0,1)$, $ p\in[1,\infty)$, the \textit{Sobolev-Slobodeckij space} $W^{\alpha, p}(0,T; X)$ consists of functions $u$ in $L^{p}(0, T; X)$ that have the additional property that 
\[\int\limits_{0}^{T}\int\limits_{0}^{T}\frac{\lVert u(t)-u(s)\rVert_{X}^{p}}{|t-s|^{1+\alpha p}}\, ds\,dt<\infty. 
\]
This is a Banach space with the norm 
\[\lVert u\rVert_{W^{\alpha, p}(0,T; X)}^{p}:=\lVert u\rVert_{L^{p}(0,T; X)}^{p}+\int\limits_{0}^{T}\int\limits_{0}^{T}\frac{\lVert u(t)-u(s)\rVert_{X}^{p}}{|t-s|^{1+\alpha p}}\, ds\,dt.
\]
 
For two vectors $\vv,\w$ in $\R^{3}$, $(\vv, \w)$ is their scalar product. 

For two integrable $\R^{3}$-valued functions $\vv(x)$ and $\w(x)$, we set 
\[\langle \vv, \w\rangle:=\int\limits_{\T^{3}}(\vv(x), \w(x))\,dx,
\]
whenever the right-hand side is finite. If both $\vv$ and $\w$ are in $L^{2}$, then $\langle \vv, \w\rangle$ is their $L^{2}$-scalar product.

For a map $\vv\in H^{2}$, $\nabla^{2}\vv$ is the tensor of second derivatives of $\vv$ with the components~$\partial^{2}_{ij}v^{k}$ for $i,\,j,\,k\in\{1,2,3\}$. The quantity 
\[
(\nabla^{2}\vv:\nabla^{2}\w)  :=\partial^{2}_{ij}v^{k}\partial^{2}_{ij}w^{k}
\] is the corresponding scalar product. Here an in what follows we sum over repeated indices. 
We denote by $Leb$ the Lebesgue measure on $\R^{3}$ and by $\mathcal L(\xi)$ is the law of a random variable $\xi$.
Finally, with $C$ we denote any constant that is independent of the index of the element in any sequence we discuss. It may depend on other parameters, and may differ from line to line. 

\subsection*{Identities for the vector product}

For vectors $\vv,\w, \uu$, and $\boldsymbol{z}$ in $\R^{3}$, we recall the following identities for the vector product:
\begin{equation}\label{triple}
(\vv\times\w, \uu)=(\uu\times \vv, \w),
\end{equation}
\begin{equation}\label{ThreeCrossProducts}
\vv\times(\w\times\uu)=(\vv,\uu)\, \w-(\vv,\w)\, \uu.
\end{equation}
In particular, 
\begin{equation}\label{CrossProductWithTheSame}
\vv\times(\vv\times\w)=-|\vv|^{2}\,\w+(\vv, \w) \,\vv.
\end{equation}
Moreover, with \eqref{ThreeCrossProducts} we have that
\begin{equation}\label{quadruple}
(\vv\times(\w\times\uu), \boldsymbol{z})=(\boldsymbol{z}\times\vv,\w\times\uu)=(\vv,\uu)(\w, \boldsymbol{z})-(\vv,\w)(\uu, \boldsymbol{z}).
\end{equation}

\subsection{Strong and weak solution and pathwise uniqueness}\label{sec:Definitions}
In this work, we consider the problem 
\begin{equation}\label{SLLG+IC}
\left\{\begin{array}{l}
 d\m=(\lambda\,\m\times\m\times(\eps^{2}\Delta^{2}\m-\Delta\m)+\m\times(\eps^{2}\Delta^{2}\m-\Delta\m))dt\\
\qquad \qquad \qquad \qquad \qquad \qquad  \qquad \qquad    \qquad    \qquad     \qquad    
+(\m\times \h)\circ dB_t,\\
\m(0)=\m_0 \\
\end{array}
\right.
\end{equation}
with non-random initial data~$\m_{0}\in H^{2}$. We assume that $\m_{0}$ satisfies the geometric constraint $|\m_{0}(x)|=1$ for all $x\in \T^{3}$. The positive parameters $\lambda$ and $\eps$ are fixed. The vector field $\h=\h(x)$ is non-random and belongs to the space $H^{2}$. The process $B_{t}$ is the standard one-dimensional Brownian motion. The equation is formulated in the Stratonovich sense.

We are looking for a solution of~\eqref{SLLG+IC} in the space~$H^{2}$. Therefore, we have to interpret in a suitable way the expressions from~\eqref{SLLG+IC} that contain the bi-Laplacian of $\m$. For $\m$~belonging to~$H^{2}$, we set 
\begin{equation}\label{eq:DefinitionOfDoubleProduct}
\langle\m\times\Delta^{2}\m, \vv\rangle:=\langle \Delta\vv\times\m, \Delta\m\rangle +2 \langle \partial_{j}\vv\times\partial_{j}\m, \Delta\m\rangle.
\end{equation}
By Lemma~\ref{prop:Sobolev}, $\m$ is in~$L^{\infty}$\! and $\partial _{j}\m$, $\partial _{j}\vv$ are in~$L^{6}$\!, and both terms on the right-hand side are well-defined. Moreover, for a smooth $\m$ identity \eqref{eq:DefinitionOfDoubleProduct} follows from \eqref{triple} and two integrations by parts. Consequently, the expression $\langle\m\times(\eps^{2}\Delta^{2}\m-\Delta\m), \vv\rangle$ is the short-hand notation for
\begin{multline*}
\langle\m\times(\eps^{2}\Delta^{2}\m-\Delta\m), \vv\rangle
=\eps^{2}\langle \Delta\vv\times\m, \Delta\m\rangle +2\eps^{2} \langle \partial_{j}\vv\times\partial_{j}\m, \Delta\m\rangle-\langle\m\times\Delta\m, \vv\rangle.
\end{multline*}

We define the quantity $\langle\m\times\m\times\Delta^{2}\m, \vv\rangle$ in the same manner. By the symmetry of the vector product \eqref{quadruple}, we have that 
\[\langle\m\times\m\times\Delta^{2}\m, \vv\rangle=\langle\m\times\Delta^{2}\m, \vv\times\m\rangle.
\]
If both $\m$ and $\vv$ belong to $H^{2}$, the function $\vv\times\m$ is again in $H^{2}$, by~Lemma~\ref{prop:Sobolev}. By virtue of \eqref{eq:DefinitionOfDoubleProduct}, we set
\begin{equation}\label{eq:DefinitionOfTripleProduct}
\langle\m\times\m\times\Delta^{2}\m, \vv\rangle:=\langle \Delta(\vv\times\m)\times\m, \Delta\m\rangle + 2\langle \partial_{j}(\vv\times\m)\times\partial_{j}\m, \Delta\m\rangle.
\end{equation}
Consequently, we use $\langle\m\times\m\times(\eps^{2}\Delta^{2}\m-\Delta\m), \vv\rangle$ as the short-hand notation for
\begin{multline*}
\langle\m\times\m\times(\eps^{2}\Delta^{2}\m-\Delta\m), \vv\rangle=\eps^{2}\langle \Delta(\vv\times\m)\times\m, \Delta\m\rangle \\
+ 2\eps^{2}\langle \partial_{j}(\vv\times\m)\times\partial_{j}\m, \Delta\m\rangle-\langle\m\times\m\times\Delta\m, \vv\rangle.
\end{multline*}

We now recall the definitions of a strong and a weak solution and of pathwise uniqueness in relation to problem~\eqref{SLLG+IC}.
\begin{mydef}\label{def:strong_solution}
Let $T>0$ be a finite time horizon. Suppose that we are given a probability space~$(\Omega, \mathfrak{F}, \pP)$ with a right-continuous complete filtration~$\{\mathfrak{F}_{t}\}$, $t\in[0,T]$ and a one-dimensional Brownian motion $B_{t}$ adapted to this filtration. 
A \textbf{stochastically strong $H^{2}$-valued solution} of~\eqref{SLLG+IC} on the time-interval $[0,T]$ with respect to $(\Omega, \mathfrak{F}, \pP)$ and~$B_{t}$ is an $\Omega\times[0,T]$-progressively measurable process $\m(\omega,t,x)$ with values in $H^{2}$ such that
\begin{enumerate}
\item the process $\m(\omega, t)$ is adapted to the augmentation $\{\mathfrak{G}_{t}\}$ of the filtration $\{\mathfrak{G}_{t}^{B}\}$ in $\mathfrak F$, where $\mathfrak{G}_{t}^{B}:=\sigma\{B_{s}, s\in[0,t]\}$, and $t\in[0,T]$; 
\item for every $t\in[0,T]$, $|\m(\omega, t, x)|=1$, $\pP\times Leb$-almost everywhere;
\item for every $\vv$ in $H^{2}$ and every $t$ in $[0,T]$, the equality
\begin{multline}\label{eq:weak formulation}
\langle \m(t),\vv\rangle=\langle \m_{0},\vv\rangle+\lambda\int\limits_{0}^{t}\langle\m\times\m\times(\eps^{2}\Delta^{2}\m-\Delta\m), \vv\rangle \,ds\\
+\int\limits_{0}^{t}\langle\m\times(\eps^{2}\Delta^{2}\m-\Delta\m), \vv\rangle \,ds
+\int\limits_{0}^{t}\langle \m\times\h,\vv\rangle \circ dB_{s}
\end{multline}
holds $\pP$-almost surely.
\end{enumerate}
\end{mydef}
We note that the solution of Definition~\ref{def:strong_solution} is analytically weak, since we use a test function to make sense of the equation. 

We now contrast Definition~\ref{def:strong_solution} to the definition of a weak martingale solution. 
\begin{mydef}\label{def:weak_solution}
We say that problem \eqref{SLLG+IC} has an \textbf{$H^{2}$-valued weak martingale solution}~$(\Omega, \mathfrak{F}, \{\mathfrak{F}_{t}\}, \pP, B_{t}, \m)$ on $[0, T]$, if there exist a probability space $(\Omega, \mathfrak{F}, \pP)$ with a right-continuous complete filtration $\{\mathfrak{F}_{t}\}$, $t\in[0,T]$, a one-dimensional Brownian motion~$B_{t}$ adapted to this filtration, and an $\Omega\times[0,T]$-progressively measurable process~$\m(\omega, t, x)$ with values in $H^{2}$ such that the conditions (2) and (3) of Definition~\ref{def:strong_solution} are satisfied for the tuple $(\Omega, \mathfrak{F}, \{\mathfrak{F}_{t}\}, \pP, B_{t}, \m)$.
\end{mydef}

The difference between a stochastically strong and a martingale solution concerns the relation of the solution to the probability space and the driving Brownian motion. For a strong solution, the probability space and the Brownian motion are prescribed in advance. 
We should be able to construct $\m$ for a given $(\Omega, \mathfrak{F}, \{\mathfrak{F}_{t}\}, \pP)$ and $B_{t}$. For a martingale solution, those objects are part of the solution we are looking for. We construct the whole tuple $(\Omega, \mathfrak{F}, \{\mathfrak{F}_{t}\}, \pP, B_{t}, \m)$ at the same time. Furthermore, condition~(1) in Definition~\ref{def:strong_solution} establishes the relation of causality between the solution and the random input $B_{t}$: The solution may depend only on the initial data and on the values of the driving Brownian motion up to time~$t$. For a martingale solution, this is not necessarily the case, since~$\{\mathfrak{F}_{t}\}$ can be strictly larger than the augmentation~$\{\mathfrak{G}_{t}\}$ of the filtration generated by $B_{t}$. 

We close this section with the definition of pathwise, or strong, uniqueness. 
\begin{mydef} \label{def:uniqueness}
\textbf{Pathwise uniqueness} holds for \eqref{SLLG+IC}, if for any two $H^{2}$-valued solutions~$\m_1$ and~$\m_2$ (martingale or strong) that are defined on the same filtered probability space~$(\Omega, \mathfrak{ F},\{\mathfrak{F}_{t}\},\pP)$, driven by the same Brownian motion and related to the same initial condition $\m_{0}$, we have
\[\pP\,\{\lVert \m_1(t) - \m_2(t)\rVert_{H^{2}}=0, \forall t\in[0,T]\}=1.
\]
\end{mydef}

\subsection{Results}
We shall prove the following two theorems concerning \eqref{SLLG+IC}. 
\begin{theorem}\label{Thm:weak}
For every finite $T>0$, problem \eqref{SLLG+IC} has an $H^{2}$-valued weak martingale solution $(\Omega, \mathfrak{F}, \{\mathfrak{F}_{t}\}, \pP, B_{t}, \m)$ on the time-interval $[0, T]$. The weak martingale solution has the following properties.
\begin{enumerate}
\item The process $\m$ satisfies the estimate
\begin{equation}\label{eq:propertyOfSolution2}
\MExp\Big[\esssup\limits_{t\in[0, T]}\lVert\m(t)\rVert_{H^{2}}^{4}\Big]<\infty.
\end{equation}
\item There exists an $\Omega\times[0,T]$-progressively measurable $L^{2}$-valued process $\mathfrak {M}$ such that 
\begin{equation}\label{eq:RepresentationInL2}
\MExp\Big[\int\limits_{0}^{T}\lVert \mathfrak {M} \rVert_{L^{2}}^{p}\,dt\Big]<\infty
\end{equation}
for all $p\in[2, \infty)$, and with 
\begin{equation}\label{eq:propertyOfSolution1}
\MExp\Big[\int\limits_{0}^{T}\langle\m\times (\eps^{2}\Delta^{2}\m - \Delta\m),\vv\rangle\,dt\Big]=\MExp\Big[\int\limits_{0}^{T}\langle\mathfrak {M},\vv\rangle\,dt\Big]
\end{equation}
for every $\vv\in L^{2}(\Omega, L^{2}(0,T; H^{2}))$.
\item For all $t\in[0, T]$, the identity 
 \begin{equation}\label{eq:StratonovichL2}
\m(t)= \m_{0}+\lambda\int\limits_{0}^{t}\m\times\mathfrak {M} \, ds+\int\limits_{0}^{t}\mathfrak {M}\, ds
+\int\limits_{0}^{t}(\m\times\h) \circ dB_{s}
 \end{equation}
 is an equality in $L^{2}$, which holds $\pP$-almost surely. The first two integrals on the right-hand side are Bochner integrals in $L^{2}$, and the last one is a Stratonovich integral.
 \item For any $\beta\in(0,1/2)$, $\m$ belongs to the space $C^{\beta}([0,T]; L^{2})$ of $\beta$-H\"older continuous functions, almost surely.
\end{enumerate}
\end{theorem}
The solution of Theorem~\ref{Thm:weak} has better regularity than required by Definitions~\ref{def:strong_solution} and~\ref{def:weak_solution}. Claim~(2) implies that the 
nonlinearity $\m\times (\eps^{2}\Delta^{2}\m - \Delta\m)$ has a well-defined $L^{2}$-representation~$\mathfrak M$. In Sections~3 and ~4, we will apply the It\^o lemma to certain functions of~$\m$. We may do that without bringing in any further argument exactly because we have an $L^{2}$-version~\eqref{eq:StratonovichL2} of the equation. 
 
We prove Theorem~\ref{Thm:weak} by closely following the work of Brzezniak et al.~\cite{Brzezniak}. The main difference concerns the functional spaces: We are working in~$H^{2}$ instead of~$H^{1}$. Due to the embedding $H^{2}\hookrightarrow L^{\infty}$, some minor shortcuts in the proof are possible. We first construct a sequence~$(\mn)$ of approximate solutions and prove that it is tight on functional spaces that are specific for our problem. Then we apply the Skorokhod representation theorem. It provides a new probability space $(\tilde\Omega, \tilde {\mathfrak{F}}, \{\tilde{\mathfrak{F}}_{t}\}, \tilde\pP)$, a Brownian motion $\tilde B_{t}$, and a limit stochastic process $\tilde\m$. We check that they indeed make up a weak martingale solution of~\eqref{SLLG+IC}. We finally show that the solution has the regularity required by Theorem~\ref{Thm:weak}. The complete proof is quite long, and can be found in~\cite{RWTHDiss}. In this paper, we merely outline the crucial steps of the proof in Section~\ref{Sec:WeakMartingaleSolution}. 

Our main contribution is the following result about strong solvability of~\eqref{SLLG+IC}.
\begin{theorem}\label{Thm:strong}
For every finite $T>0$ and for every probability space~$(\Omega, \mathfrak{F}, \pP)$ with a right-continuous complete filtration~$\{\mathfrak{F}_{t}\}$, $t\in[0,T]$ and a one-dimensional Brownian motion $B_{t}$ adapted to this filtration, there exists a stochastically strong $H^{2}$-valued solution of~\eqref{SLLG+IC} with respect to this probability space and Brownian motion. The strong solution has properties (1)-(4) listed in Theorem~\ref{Thm:weak} and is pathwise unique.
\end{theorem}
We shall prove Theorem~\ref{Thm:strong} in Section~\ref{Sec:Uniqueness}. As discussed earlier, strong solvability of~\eqref{SLLG+IC} can be derived from the pathwise uniqueness of the solution. Consequently, we verify in Section~\ref{Sec:Uniqueness} that the martingale solution of Theorem~\ref{Thm:weak} is pathwise unique, by using interpolation inequalities and product estimates. 

Another straightforward interpolation argument yields the corollary below, which we prove at the end of Section~\ref{Sec:Uniqueness}. 
\begin{corollary}\label{Cor:Interpolation}
For any $\gamma\in(0,1/8)$, the solution of~\eqref{SLLG+IC} almost surely belongs to the space~$C^{\gamma}([0,T]; C(\T^{3}))$. It therefore preserves the topology of the initial data.
\end{corollary}

\section{Construction of a weak martingale solution}\label{Sec:WeakMartingaleSolution}
For the remainder of the paper, we fix an arbitrary positive time-horizon $T$. 

The proof of Theorem \ref {Thm:weak} proceeds in four steps.
\subsection*{Step 1. Galerkin approximation} 
We approximate equation \eqref{SLLG+IC} in the finite-dimensional spaces $H_{n}:=\spn \{\e_{1},...,\e_{n}\}$, where~$\e_{k}$ are the eigenfunctions of the negative Laplacian in $H^{1}$. As in the scalar case, the functions~$\e_{k}:\T^{3}\to \R^{3}$ are smooth for every $k\in \N$. We denote by $P_{n}$ the operator of the orthogonal projection on $H_{n}$ in~$L^{2}$. 

We fix a probability space $(\Omega, \mathfrak{F}, \pP)$ with a right-continuous complete filtration $\{\mathfrak{F}_{t}\}$, $t\in[0,T]$ and a one-dimensional Brownian motion $B_{t}$ adapted to this filtration. We are looking for a process $\mn(\omega, t, x)$ in $H_{n}$ that solves the equation 
\begin{equation}\label{Approximation}
\left\{\begin{array}{l}
d\mn=P_{n}(\lambda\mn\times\mn\times(\eps^{2}\Delta^{2}\mn-\Delta\mn))\,dt\\
 \qquad\qquad\qquad+P_{n}(\mn\times(\eps^{2}\Delta^{2}\mn-\Delta\mn))\,dt+P_{n}(\mn\times \h)\circ dB_t,\\
\mn(0)=P_{n}(\m_0). \\
\end{array}
\right.
\end{equation}
Since we can write the unknown as $\mn(\omega, t, x)=\sum_{k=1}^{n}m_{n}^{k}(\omega,t)\e_{k}(x)$, equation~\eqref{Approximation} is equivalent to a system of stochastic differential equation for $m_{n}^{k}$. The system does not satisfy the conditions of the classical theorems on the existence and uniqueness of solution, since the drift term in it grows more than linearly. However, the system retains some geometric properties of the original equation: The drift term is orthogonal to the vector~$(m_{n}^{1}, ..., m_{n}^{n})$. This allows us to apply Theorem~10.6 from~\cite{Chung_Williams} and obtain the following result.
\begin{proposition}\label{existence of appr solution}
For every $n\in\N$, equation \eqref{Approximation} has a unique stochastically strong solution on $[0,T]$. 
\end{proposition}

\subsection*{Step 2. A-priori estimates and tightness of the approximating sequence}
For the sequence~$(\mn)$, we have the following a-priori estimates. 
\begin{proposition}\label{prop:TimeDependent} \mbox{}
\begin{enumerate}
\item For every $n\in\N$ and for every $t\in [0, T]$, the inequality
\begin{equation}\label{L2bound}
\lVert\mn(t)\rVert^{2}_{L^{2}}\leqslant\lVert\m_{0}\rVert^{2}_{L^{2}}
\end{equation}
holds $\pP$-almost surely.
\item For every $p\in[2,+\infty)$, there exists a constant $C$ independent of $n$ such that 
\begin{align}
&\MExp\Big[\sup\limits_{t\in[0,T]}\lVert\mn(t)\rVert_{H^{2}}^{p}\Big]\leqslant C, \label{H2bound}\\
&\MExp\Big[\Big(\int_0^T\lVert\mn\times(\eps^{2}\Delta^{2}\mn-\Delta\mn)\rVert_{L^2}^2ds\Big)^p\Big]\leqslant C,\label{precession_apriori}
\end{align}
and 
\begin{equation}\label{damping_apriori}
\MExp\Big[\Big(\int_0^T\lVert\mn\times\mn\times(\eps^{2}\Delta^{2}\mn-\Delta\mn)\rVert_{L^{2}}^2ds\Big)^p\Big]\leqslant C,
\end{equation}
for all $n\in\N$.
\item For every $p\in[2,+\infty)$ and $\alpha\in(0,1/2)$, there exists a constant $C>0$ independent of $n$ such that 
\begin{equation}\label{time_dependent_apriori}
\MExp\Big[\lVert\mn(t)\rVert_{W^{\alpha, p}(0,T;L^{2})}^{p}\Big]<C.
\end{equation}
\end{enumerate}
\end{proposition}
Estimate~\eqref{L2bound} follows immediately from~\eqref{Approximation}, estimates \eqref{H2bound}--\eqref{damping_apriori} are derived from the evolution equation for the energy. Using all of them, we estimate all terms in equation~\eqref{Approximation} in the space~$W^{\alpha, p}(0,T;L^{2})$ and obtain~\eqref{time_dependent_apriori}.

For $p\alpha >1$, the space $L^{p}(0,T;H^{2})\cap W^{\alpha,p}(0,T;L^2)$ is compactly embedded into the space $L^{p}(0,T;H^{s})\cap C([0,T]; H^{-\beta})$, for any $s<2$ and $\beta>0$ (\cite{LionsQuelquesMethodes}, Chapter 5 and~\cite{FlandoliGatarek}, Section 2). This gives our key tightness result.
\begin{proposition}\label{Prop:Tightness}
For every $p\in(2,+\infty)$, $0\leqslant s<2$ and $\beta>0$, the sequence $(\mn(t))$ is tight on the space $L^{p}(0,T;H^{s})\cap C([0,T]; H^{-\beta})$.
\end{proposition}
For the remainder of the paper, we fix the values of the parameters in Proposition~\ref{Prop:Tightness}. We take $p=4$, $s=3/2$, $\beta=1/2$. We denote the corresponding path space by 
\[\X:=L^{4}(0,T;H^{3/2})\cap C([0,T]; H^{-1/2}). 
\]
\subsection*{Step 3. Passing to the limit}

We consider the sequence of pairs~$((\m_{n}, B_{t}))$ on the product space $\X\times C([0,T])$. The second element in the pair is the same for all $n$. By Proposition~\ref{Prop:Tightness}, this sequence is tight on the space $\X\times C([0,T])$. By the Prokhorov theorem, the corresponding sequence of laws $\mathcal{L}((\m_{n}, B_{t}))$ contains a convergent subsequence. From now on, we work with this subsequence without relabelling it. We denote the limiting probability measure on $\X\times C([0,T])$ by $\mathcal{P}$. Since the space $\X\times C([0,T])$ is separable, we may apply to $((\m_{n}, B_{t}))$ the Skorokhod representation theorem. We summarize its conclusions in the proposition below. 
\begin{proposition}\label{Prop:Representation}
There exists a probability space $(\tilde{\Omega}, \tilde{\mathfrak F}, \tilde \pP)$ and on it, a sequence \linebreak $((\tilde\m_{n}, \tilde B_n(t)))$ with the following properties. 
\begin{enumerate}
\item For every $n$, $\tmn$ takes values in the space $\X$ and $\tilde B_{n}$ - in the space $C([0,T])$;
\item the sequence $((\tilde\m_{n}, \tilde B_n(t)))$ converges in $\X\times C([0,T])$, $\tilde \pP$ - almost surely, to an element $(\tilde \m, \tilde B_{t})$;
\item for every $n$, $\mathcal{L}((\tilde \m_{n}, \tilde B_{n}(t)))=\mathcal{L}((\m_{n}, B_{t}))$ and $\mathcal{L}((\tilde \m, \tilde B_{t}))=\mathcal{P}$.
\end{enumerate}
\end{proposition}

The pair $(\tilde \m, \tilde B_{t})$, the space $(\tilde{\Omega}, \tilde{\mathfrak F}, \tilde \pP)$, and the filtration $\tilde{\mathfrak F}_{t}$, which is the augmentation of $\sigma ((\tilde \m(s),\tilde B_{s}), s\in [0, t] )$, is the candidate for a weak martingale solution of~\eqref{SLLG+IC}. We should check that $\tilde \m$ solves~\eqref{SLLG+IC} driven by~$\tilde B_{t}$.

Since the laws of $\mn$ and $\tmn$ coincide, the a-priori estimates of Propositon~\ref{prop:TimeDependent} are valid for the sequence $\tmn$. Therefore, $\tilde \m$ satisfies \eqref{eq:propertyOfSolution1} and \eqref{eq:propertyOfSolution2} as the pointwise limit of~$\tmn$. Moreover, these estimates guarantee that the sequence $(\tmn\times(\eps^{2}\Delta^{2}\tmn-\Delta\tmn))$ subconverges on the space $L^{p}(\tilde \Omega, L^{p}(0,T; L^{2}))$ for every $p\in[2, \infty)$. The limit $\tilde{\mathfrak M}$ is the representation of $\tilde \m\times(\eps^{2}\Delta^{2}\tilde \m-\Delta\tilde \m)$ in the sense of item~(2) from Theorem~\ref{Thm:weak}.

Due to the equality of laws, $\tilde B_n(t)$ is a Brownian motion for every $n$, and the processes~$\tmn$ solve the approximating system~\eqref{Approximation} driven by $\tilde B_n(t)$. Now we may pass to the limit $n\to \infty$ in this equation, using the almost sure convergence of~$((\tmn, \tilde B_n(t)))$ to~$(\tilde \m, \tilde B_{t})$ and the a-priori estimates for~$\tmn$. This ensures that $\tilde\m$ is a weak martingale solution of~\eqref{SLLG+IC} as defined in~\eqref{eq:weak formulation} and satisfies \eqref{eq:propertyOfSolution2}. Since we also have the representation process $\tilde {\mathfrak M}$ satisfying~\eqref{eq:RepresentationInL2} and~\eqref{eq:propertyOfSolution1}, the pair $(\tilde \m, \tilde B_{t})$ solves equation~\eqref{SLLG+IC} in the $L^{2}$-sense. 

\subsection*{Step 4. Improved analytic properties of the weak martingale solution}   
It remains to show that the process $\tilde \m $ satisfies condition~(2) of Definition~\ref{def:strong_solution} and condition~(4) of Theorem~\ref{Thm:weak}.
 
To prove that $|\tilde \m|=1$ for almost all $\tilde \omega$, $x$ and $t$, we show that for every $\phi \in C_{0}^{\infty}(D; \R)$ there holds 
\[d\langle \tilde \m, \phi\tilde \m\rangle=0.
\]
This follows from the It\^{o} lemma~\cite{Pardoux_SPDE}.
Therefore, 
\[\int\limits_{D}\phi|\tilde\m(t)|^{2}\,dx=\int\limits_{D}\phi|\tilde\m_{0}|^{2}\,dx=\int\limits_{D}\phi\,dx,
\]
and the claim follows.

The H\"older-continuity of $\tilde \m(t)$ is verified via the Kolmogorov test. This concludes the proof of Theorem~\ref{Thm:weak}.

\section{Pathwise uniqueness}\label{Sec:Uniqueness}
In this section, we prove Theorem~\ref{Thm:strong}. We start with the following simple observation.
\begin{proposition}
If the approximating sequence~$(\m_{n})$ converges in probability on the original probability space $(\Omega, \mathfrak{F}, \pP)$ with respect to the topology of~$\X$ to some~$\m:\Omega\to \X$, then~$\m$ is a stochastically strong solution of~\eqref{SLLG+IC}.
\end{proposition}
\begin{proof}
If $(\mn)$ converges to $\m$ in probability, there exists a subsequence of~$(\mn)$ that converges to $\m$ $\pP$-almost surely. For this subsequence, which we do not relabel, we can repeat verbatim the arguments that we used in Steps~3 and~4 of Section~\ref{Sec:WeakMartingaleSolution} for~$(\tmn)$ and~$\tilde \m$. It follows that $\m$ solves~\eqref{SLLG+IC} driven by $B_{t}$ and has properties (1)--(4) listed in Theorem~\ref{Thm:weak}. Furthermore, $\m$ satisfies the adaptiveness property~(1) of Definition~\ref{def:strong_solution}. Indeed, almost sure convergence in $\X$ implies that for every $t\in[0,T]$, $\m(t)$ is an almost sure limit of $\mn(t)$ in the $H^{-1/2}$-norm. For every $n\in\N$, $\mn(t)$ is a stochastically strong solution of~\eqref{Approximation}, and is therefore adapted to~$\{\mathfrak{G}_{t}\}$. By definition this means that, for every $t\in[0,T]$ the random variable $\mn(t)$ is $\mathfrak G_{t}$-measurable. Therefore, $\m(t)$ is $\mathfrak G_{t}$-measurable as an almost sure limit of $\mathfrak G_{t}$-measurable functions. But this means that the process $\m$ is adapted to~$\{\mathfrak{G}_{t}\}$ and is thus a strong solution of~\eqref{SLLG+IC} on~$(\Omega, \mathfrak{F}, \{\mathfrak{F}_{t}\},\pP)$. 
\end{proof}
In order to improve the weak convergence of~$(\mn)$ to the convergence in probability we use the following variant of the Yamada-Watanabe argument that is due to Gy\"{o}ngy and Krylov. 
\begin{lemma}[\cite{Gyongy_Krylov}, Lemma 1.1]\label{Lem:Gyongy_Krylov}
Let $(Z_n)$ be a sequence of random elements with values in a Polish space $E$ equipped with the Borel sigma-algebra. Then $(Z_n)$ converges in probability to an $E$-valued random element if and only if the following condition holds. For every pair of subsequences $(Z_{n_i})$, $(Z_{m_j})$, there exists a further subsequence $(V_k)$ of the product sequence $((Z_{n_i},Z_{m_i}))$ that converges weakly to a random element $V$ supported on the diagonal of $E\times E$.
\end{lemma}
To apply Lemma~\ref{Lem:Gyongy_Krylov}, we have to check that the weak martingale solution of~\eqref{SLLG+IC} is pathwise unique. 
Indeed, in our case, $(\mn)$ plays the role of $(Z_n)$, and $\X$ that of $E$. Let  $(\m_{n_{i}})$ and $(\m_{m_{j}})$ be two subsequences of $(\mn)$. Then, the product subsequence~$((\m_{n_{i}}, \m_{m_{i}}))$ is tight on $\X\times \X$, by Proposition~\ref{Prop:Tightness}. On the one hand, this gives us a subsequence~$(V_{k})$ of $((\m_{n_{i}}, \m_{m_{i}}))$ that converges weakly to a limit~$V$. Without loss of generality, we assume that the whole sequence $((\m_{n_{i}}, \m_{m_{i}}))$ converges weakly. On the other hand, the sequence of triples~$((\m_{n_{i}}, \m_{m_{i}}, B_{t}))$ converges weakly on the space $\X\times\X\times C([0,T])$. We represent it via a pointwise convergent sequence $((\tm_{1, n_{i}}, \tm_{2, m_{i}}, \tilde B_{i}(t)))$, again using the Skorokhod representation theorem. 

The results of Steps 3 and 4 of Section~\ref{Sec:WeakMartingaleSolution} apply to the sequences $((\tm_{1, n_{i}},  \tilde B_{i}(t)))$ and~$((\tm_{2,m_{i}}, \tilde B_{i}(t)))$. We thus obtain two weak martingale solutions $\tm_{1}$, $\tm_{2}$ of~\eqref{SLLG+IC}. The first is the pointwise limit of the sequence~$(\tm_{n_{i}})$, the second is the pointwise limit of the sequence~$(\tm_{m_{i}})$. The processes $\tm_{1}$ and~$\tm_{2}$ are thus defined on the same filtered probability space and are driven by the same Brownian motion. The random element~$V$ has the same law as the pair $(\tilde\m_{1}, \tilde \m_{2})$, again due to the Skorokhod representation theorem. To use Lemma~\ref{Lem:Gyongy_Krylov}, we have to show that $\tilde \m_{1}$ and $\tilde \m_{2}$ coincide with probability one. This is equivalent to the pathwise uniqueness of the weak martingale solution of~\eqref{SLLG+IC}. 

The objective of this section is therefore to prove the following proposition.
\begin{proposition}\label{prop:PathwiseUniqueness}
Pathwise uniqueness holds for equation \eqref{SLLG+IC}.
\end{proposition}

We now present the tools that we use in the proof. We first recall two interpolation inequalities for the norms $\lVert u\rVert_{L^\infty}$ and $\lVert \nabla u\rVert_{L^4}$ of a function $u\in H^{2}$. We formulate them for real-valued functions, since we apply them to the components of $\tm$. By the Sobolev embedding, both these quantities are controlled by the norm $\lVert u\rVert_{H^{2}}$, but this estimate is too crude for our proof of uniqueness. We need bounds that include a portion of norms weaker than the $H^{2}$-norm. 


\begin{lemma}[Agmon's inequality, \cite{AgmonLecturesOnElliptic}]
There exists a constant $C$ such that for all~$u$ in~$H^2$, there holds
\begin{equation}\label{Agmon}
\lVert u\rVert_{L^\infty}\leqslant  C\, \lVert u\rVert_{L^{2}}^{1/4}\lVert u\rVert_{H^2}^{3/4}.
\end{equation}
\end{lemma}

Agmon's inequality is the only detail in our proof that is slightly different in the two-dimensional case. It holds in dimension two as well, but there, it reads
\[\lVert u\rVert_{L^\infty}\leqslant  C\, \lVert u\rVert_{L^{2}}^{1/2}\lVert u\rVert_{H^2}^{1/2}.
\]
 
\begin{lemma}[Product estimate for $\lVert \nabla u\rVert_{L^4}$]\label{Lem:Gagliardo}
There exists a constant C such that for all $u$ in $H^2$, there holds
\begin{equation}\label{GN}
\lVert \nabla u\rVert_{L^4}\leqslant C\lVert  u\rVert_{L^2}^{1/4}\lVert  u\rVert_{H^2}^{3/4}.
\end{equation}
\end{lemma}
\begin{proof}
Interpolation between $L^{2}$ and $L^{6}$ gives 
\[\lVert \nabla u\rVert_{L^4}^{2}\leqslant C\lVert \nabla u\rVert_{L^2}\lVert \nabla u\rVert_{L^6}. 
\]
With integration by parts and H\"{o}lder's inequality we see that 
\[\lVert \nabla u\rVert_{L^2}\leqslant C \lVert u\rVert_{L^2}^{1/2}\lVert \Delta u\rVert_{L^2}^{1/2}\leqslant C \lVert u\rVert_{L^2}^{1/2}\lVert u\rVert_{H^2}^{1/2}.
\]
By the Sobolev embedding,
\[\lVert \nabla u\rVert_{L^6}\leqslant C\lVert u\rVert_{H^2}.
\]
Combining the estimates above, we obtain~\eqref{GN}.
\end{proof}

Next we present two technical facts about solutions of~\eqref{SLLG+IC}.
\begin{proposition}\label{Prop:IdentitiesForM}
Any weak martingale solution of \eqref{SLLG+IC} from Theorem~\ref{Thm:weak} has the following properties.
\begin{enumerate}
\item For every $t\in [0, T]$, there holds 
\[(\tilde \m(t) \cdot \nabla)\, \tilde \m(t)=0,
\] 
for almost all $(\tilde \omega, x)\in\tilde \Omega\times \T^{3}$.
\item For any $\vv\in H^{2}$ and almost all $(\tilde \omega, t)\in \tilde \Omega\times [0, T]$, there holds 
\begin{equation}\label{eq:TripleProdDifferently}
\langle \tilde \m\times \tilde \m \times \Delta^{2}\tilde \m, \vv\rangle
=-\langle \Delta \vv, \Delta \tilde \m\rangle +\langle |\Delta \tilde \m |^{2}\tilde \m, \vv\rangle -2\langle \nabla \tilde \m\otimes \nabla \tilde \m, \nabla^{2}(\vv, \tilde\m)\rangle.
\end{equation}
\end{enumerate}
\end{proposition} 
Above, $(\tilde \m \cdot \nabla )\,\tilde \m$ is a vector in $\R^{3}$ with the components $(\tilde \m, \partial_{k}\tilde \m)$ for $k\in\{1,2,3\}$. The expression $\langle \nabla \tilde \m\otimes \nabla \tilde \m, \nabla^{2}(\vv, \tilde\m)\rangle$ stands for $\langle (\partial_{j} \tilde \m, \partial_{k} \tilde \m),  \partial^{2}_{jk}(\vv, \tilde\m)\rangle$, the $L^{2}$-product of two $3\times 3$ matrices. 
\begin{proof}
\textbf{Proof of (1)}. 
Since the process $\tilde \m $ belongs to $H^{2}$,  $|\tilde\m|^{2}$ belongs to~$H^{2}$ as well, by Lemma~\ref{prop:Sobolev}. Again by Lemma~\ref{prop:Sobolev}, $\partial_{k}|\tilde \m|^{2}$ is a process with values in $L^{6}$. Moreover, for almost every $x\in \T^{3}$ and for every $t\in[0,T]$, we have that $|\tilde \m(t,x)|^{2}=1$, almost surely. By the chain rule, we obtain for $k\in\{1,2,3\}$ that
\[2(\tilde \m(t),\partial_{k}\tilde \m(t))=\partial_{k}|\tilde \m(t)|^{2}=\partial_{k}1=0, 
\]
for every $t\in [0, T]$, $\tilde\pP\times Leb$-almost everywhere.

\textbf{Proof of (2)}. We write the quantity $\langle \tilde \m\times \tilde \m \times \Delta^{2}\tilde \m, \vv\rangle $ according to Definition~\eqref{eq:DefinitionOfTripleProduct} and then transform it using identity~\eqref{quadruple} and the identities 
\[|\tilde \m|^{2}=1,\ (\tilde \m \cdot \nabla) \tilde \m=0
\]
that hold $\tilde\pP\times Leb$-almost everywhere. We obtain that
\begin{align*}
\langle \tilde \m\times \tilde \m \times \Delta^{2}\tilde \m, \vv\rangle =&\, \langle (\Delta \vv\times \tilde\m )\times \tilde\m, \Delta \tilde\m\rangle +2\langle (\partial_{j}\vv\times \partial_{j}\tilde\m )\times \tilde\m,\Delta \tilde\m\rangle \\
&+\langle (\vv\times\Delta\tilde \m)\times\tilde \m, \Delta\tilde \m\rangle+2\langle (\partial_{j}\vv\times \tilde\m)\times\partial_{j}\tilde\m, \Delta \tilde\m\rangle\\
&+2\langle (\vv\times \partial_{j}\tilde \m)\times \partial_{j}\tilde \m, \Delta\tilde \m\rangle\\
=&-\langle \Delta \vv, \Delta\tilde \m\rangle-\langle |\nabla \tilde \m|^{2}\tilde \m, \Delta\vv\rangle\\
&+2\langle (\tilde \m, \partial_{j}\vv)\partial_{j}\tilde \m, \Delta \tilde \m\rangle +\langle |\Delta \tilde \m|^{2}\tilde \m, \vv\rangle \\
&-2\langle (|\nabla \tilde \m|^{2}\partial_{j}\vv, \partial_{j}\tilde\m\rangle +2\langle (\vv, \partial_{j}\tilde\m)\partial_{j}\tilde \m, \Delta\tilde \m\rangle -\langle |\nabla \tilde \m|^{2}\vv, \Delta\tilde \m\rangle \\
=&-\langle \Delta \vv, \Delta\tilde \m\rangle+\langle |\Delta \tilde \m|^{2}\tilde \m, \vv\rangle\\
&-\langle |\nabla \tilde \m|^{2},\Delta(\vv, \tilde \m)\rangle +2\langle \partial_{j}(\vv,\tilde\m)\partial_{j}\tilde \m, \Delta\tilde \m\rangle.
\end{align*}
We integrate by parts twice in the last term and obtain
\begin{align*}
2\langle \partial_{j}(\vv,\tilde\m)\partial_{j}\tilde \m, \Delta\tilde \m\rangle&=-2\langle (\partial_{j}\tilde\m,\partial_{k}\tilde \m), \partial_{jk}^{2}(\vv, \tilde \m)\rangle -\langle \partial_{j}(\vv, \tilde\m ),\partial_{j}|\nabla \tilde \m|^{2}\rangle\\
&=-2\langle (\partial_{j}\tilde\m,\partial_{k}\tilde \m), \partial_{jk}^{2}(\vv, \tilde \m)\rangle +\langle \Delta(\vv, \tilde\m ),|\nabla \tilde \m|^{2}\rangle. 
\end{align*}
This gives \eqref{eq:TripleProdDifferently}.
\end{proof}

Now we have everything in place to prove Proposition \ref{prop:PathwiseUniqueness}.
\begin{proof}[Proof of Proposition \ref{prop:PathwiseUniqueness}]
We first observe that \eqref{SLLG+IC} is equivalent to the following It\^o equation
\begin{equation}\label{LLG_Ito}
\m(t)= \m_{0}+\lambda\int\limits_{0}^{t}\m\times\mathfrak {M} \, ds+\int\limits_{0}^{t}\mathfrak {M}\, ds
+\frac{1}{2}\int\limits_{0}^{t}((\m\times \h)\times\h)\,ds+\int\limits_{0}^{t}(\m\times\h)dB_{s},
 \end{equation}
which holds in the $L^{2}$-sense, due to Theorem~\ref{Thm:weak}. 
Indeed, the function 
\[F: H^{2}\rightarrow H^{2}, \ \m\mapsto \m\times \h,
\]
in the stochastic term in~\eqref{LLG} is linear in $\m$. Therefore, the Stratonovich-It\^{o} correction term is given simply by 
\[\frac{1}{2}F(F(\m))=\frac{1}{2}((\m\times \h)\times\h).
\]

Let now $\tilde\m_{1}$ and $\tilde\m_{2}$ be two weak martingale solutions of~\eqref{LLG_Ito} that are defined on the same probability space and for the same Brownian motion~$\tilde B_{t}$. To prove pathwise uniqueness, it suffices to show that $\tilde \m_{1}(t)$ and $\tilde\m_{2}(t)$ coincide for all $t\in[0,T]$ as elements of~$L^{2}$ (and not of~$H^{2}$, as in Definition~\ref{def:uniqueness}), i.e., that
\begin{equation}
\label{eq:UniqInL2}
\pP\,\{\lVert \m_1(t) - \m_2(t)\rVert_{L^{2}}=0, \forall t\in[0,T]\}=1.
\end{equation}
To this end, we derive an equation for the norm $\frac{1}{2}\lVert \tilde\m_{1}(t) - \tilde \m_{2}(t)\rVert_{L^{2}}^{2}$ by applying to it the It\^o lemma. Since the processes $\tilde\m_{1}$ and $\tilde\m_{2}$ are $H^{2}$-valued and satisfy equation~\eqref{LLG_Ito} in the $L^{2}$-sense, we may use the It\^{o} lemma in the formulation of Pardoux~\cite{Pardoux_SPDE}. This yields the equation
\begin{align*}
d\frac{1}{2}\lVert \tilde\m_{1}(t)-\tilde \m_{2}(t)\rVert_{L^{2}}^{2}=&\,\lambda\langle \tilde\m_{1}-\tilde \m_{2}, \tilde\m_{1}\times\tilde{\mathfrak M}_{1}\rangle \,dt\\
&-\lambda\langle \tilde\m_{1}-\tilde \m_{2}, \tilde\m_{2}\times\tilde{\mathfrak M}_{2}\rangle \,dt\\
&+\langle \tilde\m_{1}-\tilde \m_{2},\tilde{\mathfrak M}_{1}\rangle \,dt
-\langle \tilde\m_{1}-\tilde \m_{2}, \tilde{\mathfrak M}_{2}\rangle \,dt\\
&+\frac{1}{2}\langle \tilde\m_{1}-\tilde \m_{2}, (\tilde\m_{1}\times\h)\times \h-(\tilde\m_{2}\times\h)\times \h\rangle \,dt\\
&+\frac{1}{2}\lVert \tilde\m_{1}\times \h-\tilde \m_{2}\times \h\rVert_{L^{2}}^{2} \,dt\\
&+\langle \tilde\m_{1}-\tilde \m_{2},(\tilde\m_{1}-\tilde \m_{2})\times\h\rangle \,d\tilde B_{t}.
\end{align*}
We see immediately that the stochastic term is identically zero. Moreover, the two terms containing the field $\h$ cancel each other, due to the antisymmetry of the cross product. Thus, the equation on $\frac{1}{2}\lVert \tilde\m_{1}(t) - \tilde \m_{2}(t)\rVert_{L^{2}}^{2}$ is deterministic. For convenience, we rewrite it in terms of the function  
\[
\uu(t):=\tilde\m_{1}(t)-\tilde \m_{2}(t).
\]
We have that for all $t\in[0, T]$ and almost every $\tilde\omega\in\tilde \Omega$, there holds
\begin{multline}\label{eq:Difference1}
\frac{1}{2}\lVert \uu(t)\rVert_{L^{2}}^{2}(\tilde \omega)=\int\limits_{0}^{t}\,\lambda\langle \uu, \tilde\m_{1}\times\tilde{\mathfrak M}_{1}\rangle (\tilde \omega)\,ds
-\int\limits_{0}^{t}\lambda\langle \uu, \tilde\m_{2}\times\tilde{\mathfrak M}_{2}\rangle (\tilde \omega) \,ds\\
+\int\limits_{0}^{t}\langle \uu, \tilde{\mathfrak M}_{1}-\tilde{\mathfrak M}_{2}\rangle(\tilde \omega) \,ds.
\end{multline}
Note that $\uu$, and consequently $\uu\times \tilde \m_{1}$ and $\uu\times \tilde \m_{2}$, belong to $L^{2}(\tilde \Omega, L^{2}(0,T; H^{2}))$, so we may use representation formula~\eqref{eq:propertyOfSolution1} for $\tilde {\mathfrak M}_{1}$ and $\tilde {\mathfrak M}_{2}$. With it, we obtain that~\eqref{eq:Difference1} is equivalent to equation
\begin{align*}
\frac{1}{2}\lVert \uu(t)\rVert_{L^{2}}^{2}=\int\limits_{0}^{t}&\,\lambda\langle \uu, \tilde\m_{1}\times\tilde\m_{1}\times(\eps^{2}\Delta^{2}\tilde\m_{1}-\Delta\tilde\m_{1})\rangle \,ds\\
&-\int\limits_{0}^{t}\lambda\langle \uu, \tilde\m_{2}\times\tilde\m_{2}\times(\eps^{2}\Delta^{2}\tilde\m_{2}-\Delta\tilde\m_{2})\rangle \,ds\\
&+\int\limits_{0}^{t}\langle \uu, \tilde\m_{1}\times(\eps^{2}\Delta^{2}\tilde\m_{1}-\Delta\tilde\m_{1})-\tilde\m_{2}\times(\eps^{2}\Delta^{2}\tilde\m_{2}-\Delta\tilde\m_{2})\rangle \,ds,
\end{align*}
which holds $\tilde\pP$-almost surely. 
We transform the right-hand side with the help of \eqref{eq:TripleProdDifferently} and arrive at our key equation  
\begin{subequations}\label{2difference}
\begin{align}
\frac{1}{2}\lVert \uu\rVert_{L^{2}}^{2}(t)=&\,-\int\limits_{0}^{t}\lambda \eps^2\langle\Delta\uu,\Delta\uu\rangle\,ds\label{line1}\\
&+\int\limits_{0}^{t}\lambda\eps^2\langle |\Delta\tilde\m_1|^2\tilde\m_1-|\Delta\tilde\m_2|^2\tilde\m_2,\uu\rangle\,ds\label{line2}\\
&-\int\limits_{0}^{t}\lambda\eps^2\langle\nabla\tilde\m_1\otimes\nabla\tilde\m_1:\nabla^2(\tilde\m_1,\uu)\rangle\,ds\label{line2.5}\\
& +\int\limits_{0}^{t}\lambda\eps^2\langle\nabla\tilde\m_2\otimes\nabla\tilde\m_2:\nabla^2(\tilde\m_2,\uu)\rangle\,ds\label{line3}\\
&+\int\limits_{0}^{t}\eps^2\langle \Delta\tilde\m_1,\tilde\m_1\times\Delta\uu\rangle\,ds-\int\limits_{0}^{t}\eps^2\langle \Delta\tilde\m_2,\tilde\m_2\times\Delta\uu\rangle\,ds\label{line5}\\
&+\int\limits_{0}^{t}2\eps^2\langle \Delta\tilde\m_1,\partial_j\tilde\m_1\times\partial_j\uu\rangle\,ds-\int\limits_{0}^{t}2\eps^2\langle \Delta\tilde\m_2,\partial_j\tilde\m_2\times\partial_j\uu\rangle\,ds\label{line6}\\
&-\int\limits_{0}^{t}\lambda\langle\tilde\m_1\times\tilde\m_1\times\Delta\tilde\m_1- \tilde\m_2\times\tilde\m_2\times\Delta\tilde\m_2,\uu\rangle\,ds\label{line7}\\
&-\int\limits_{0}^{t}\langle\tilde\m_1\times\Delta\tilde\m_1-\tilde\m_2\times\Delta\tilde\m_2,\uu\rangle\,ds.\label{line8}
\end{align}
\end{subequations}


Our next aim is to bound $\lVert\uu(\tilde \omega, t)\rVert_{L^2}^2$ from above by $C\int_{0}^{t}\lVert\uu(\tilde \omega, s)\rVert_{L^2}^2\,ds$ for all $\tilde\omega$ for which~\eqref{2difference} holds, in order to apply the Gronwall lemma afterwards. To this end, we first estimate the terms on the right-hand side of \eqref {2difference} line by line, using H\"{o}lder's inequality and inequalities~\eqref{Agmon}, \eqref{GN}, by the quantities of the form $C\lVert\uu\rVert_{L^2}^{p_{k}}\lVert\uu\rVert_{H^2}^{2-p_{k}}$, for several different $p_{k}\in(0,2]$. The constant~$C$ may depend on the $H^2$-norms of~$\tilde\m_1(\tilde \omega)$ and $\tilde\m_2(\tilde \omega)$. In the computations below, we omit the integration in time because our manipulations are in the space variables only. 

The term on the right-hand side of \eqref {line1} is simply
\[-\lambda \eps^2\lVert\Delta\uu\rVert_{L^2}^2=-\lambda \eps^2\lVert\uu\rVert_{\dot{H}^2}^2.
\]
Note carefully that we use the homogeneous norm here. 

In \eqref{line2}, we use H\"{o}lder's inequality in the first step and~\eqref{Agmon} in the second. We obtain that
\begin{align*}
\lambda\eps^2\langle |\Delta\tilde\m_1|^2&\tilde\m_1- |\Delta\tilde\m_2|^2\tilde\m_2, \uu\rangle\\
&= \lambda\eps^2\langle |\Delta\tilde\m_1|^2\uu,\uu\rangle+\lambda\eps^2\langle (\Delta\uu,\Delta\tilde\m_1+\Delta\tilde\m_2)\tilde\m_2,\uu\rangle\\
&\leqslant \,\lambda\eps^2\lVert\tilde\m_1\rVert_{H^2}^2\lVert\uu\rVert_{L^\infty}^2+\lambda\eps^2(\lVert\tilde\m_1\rVert_{H^2}+\lVert\tilde\m_2\rVert_{H^2})\lVert\tilde\m_2\rVert_{L^\infty}\lVert\uu\rVert_{H^2}\lVert\uu\rVert_{L^\infty}
\\
&\leqslant \,C\eps^2\lambda(\lVert\uu\rVert_{L^2}^{1/2}\lVert\uu\rVert_{H^2}^{3/2}+\lVert\uu\rVert_{L^2}^{1/4}\lVert\uu\rVert_{H^2}^{7/4}).
\end{align*}
We transform the quantities in lines \eqref{line2.5} and \eqref{line3} together to 
\begin{multline}\label{line2.75}
-\lambda\eps^2\langle \nabla\uu\otimes\nabla\tilde\m_2:\nabla^2(\tilde\m_1,\uu)\rangle-\lambda\eps^2\langle \nabla\tilde\m_1\otimes\nabla\uu:\nabla^2(\tilde\m_1,\uu)\rangle\\
-\lambda\eps^2\langle \nabla\tilde\m_2\otimes\nabla\tilde\m_2:\nabla^2(|\uu|^2)\rangle.
\end{multline}
Since the first two terms in \eqref{line2.75} are essentially of the same form, it suffices to estimate one of them. In what follows, we omit the common prefactor $(-\lambda \eps^{2})$. We apply H\"{o}lder's inequality, \eqref{Moser1}, and \eqref{GN}, and obtain 
\begin{align*}
\langle \nabla\uu\otimes\nabla\tilde\m_2 :\nabla^2(\tilde\m_1,\uu)\rangle&\leqslant \lVert\nabla\uu\otimes\nabla\tilde\m_2\rVert_{L^2}\lVert\nabla^2(\tilde\m_1,\uu)\rVert_{L^2}
\\
&\leqslant C \lVert\nabla\tilde\m_2\rVert_{L^4}\lVert\nabla\uu\rVert_{L^4}\left(\lVert\tilde\m_1\rVert_{H^2}\lVert\uu\rVert_{L^\infty}+\lVert\uu\rVert_{H^2}\lVert\tilde\m_1\rVert_{L^\infty}\right)
\\
&\leqslant C(\lVert\uu\rVert_{L^2}^{1/2}\lVert\uu\rVert_{H^2}^{3/2}+\lVert\uu\rVert_{L^2}^{1/4}\lVert\uu\rVert_{H^2}^{7/4}).
\end{align*}
We turn our attention to the last term in \eqref{line2.75}. With H\"{o}lder's inequality, \eqref{GN}, and~\eqref{Agmon}, we obtain that 
\begin{align*}
\langle \nabla\tilde\m_2\otimes\nabla\tilde\m_2:\nabla^2(|\uu|^2)\rangle&=2\langle \nabla\tilde\m_2\otimes\nabla\tilde\m_2:((\partial_k\uu, \partial_j\uu)+(\uu,\partial^2_{j,k}\uu))\rangle
\\
&\leqslant \lVert\nabla\tilde\m_2\rVert_{L^4}^2(\lVert\nabla\uu\rVert_{L^4}^2+\lVert\uu\rVert_{L^\infty}\lVert\uu\rVert_{H^2})
\\
&\leqslant C\lVert \uu\rVert_{L^2}^{1/4}\lVert \uu\rVert_{H^2}^{7/4}.
\end{align*}
For line \eqref{line5} we obtain, with H\"{o}lder's inequality and \eqref{Agmon}, that 
\begin{align*}
\eps^2\langle \Delta\tilde\m_1,\tilde\m_1\times\Delta\uu\rangle- \eps^2\langle \Delta\tilde\m_2,\tilde\m_2\times\Delta\uu\rangle
&=\eps^2\langle\Delta\uu,\tilde\m_1\times\Delta\uu\rangle
+\eps^2\langle \Delta\tilde\m_2,\uu\times\Delta\uu\rangle
\\
&=\eps^2\langle \Delta\tilde\m_2,\uu\times\Delta\uu\rangle\\
&\leqslant \eps^2 C( \lVert\tilde\m_2\rVert_{H^2})\lVert \uu\rVert_{H^2}\lVert \uu\rVert_{L^\infty}\\
&\leqslant \eps^2 C\lVert \uu\rVert_{L^2}^{1/4}\lVert \uu\rVert_{H^2}^{7/4}. 
\end{align*}
Line \eqref{line6} is estimated by H\"{o}lder's  inequality and \eqref{GN}:
\[
\eps^2\langle \Delta\uu,\partial_j\tilde\m_1\times\partial_j\uu\rangle\leqslant \eps^2C\lVert\tilde\m_1\rVert_{H^2}\lVert \uu\rVert_{H^2}\lVert \nabla\uu\rVert_{L^4}
\leqslant \eps^2C\lVert \uu\rVert_{L^2}^{1/4}\lVert \uu\rVert_{H^2}^{7/4}.
\]
We transform line \eqref{line7} to 
\[-\lambda\lVert\nabla\uu\rVert_{L^2}^2-\lambda\langle \Delta\uu,(\tilde\m_1,\uu)\tilde\m_1\rangle-\lambda\langle\Delta\tilde\m_2,|\uu|^2\tilde\m_1\rangle -\lambda\langle\Delta\tilde\m_2,(\tilde\m_2,\uu)\uu\rangle
\]
and conclude, by H\"{o}lder's inequality and \eqref{Agmon}, that 
\begin{align*}
-\lambda\lVert&\nabla\uu\rVert_{L^2}^2-\lambda\langle \Delta\uu,(\tilde\m_1,\uu)\tilde\m_1\rangle-\lambda\langle\Delta\tilde\m_2,|\uu|^2\tilde\m_1\rangle -\lambda\langle\Delta\tilde\m_2,(\tilde\m_2,\uu)\uu\rangle\\
&\leqslant -\lambda\lVert\nabla\uu\rVert_{L^2}^2+C(\lVert \uu\rVert_{H^2}\lVert \uu\rVert_{L^2}+\lVert \uu\rVert_{L^\infty}^2)\\
&\leqslant -\lambda\lVert\nabla\uu\rVert_{L^2}^2+C(\lVert \uu\rVert_{H^2}\lVert \uu\rVert_{L^2}+\lVert \uu\rVert_{L^2}^{1/4}\lVert \uu\rVert_{H^2}^{7/4}).
\end{align*}

The integrand in the last line \eqref{line8} is equal to the quantity $\langle\tilde\m_2\times\Delta\uu,\uu\rangle$, for which we have the estimate
\[\langle\tilde\m_2\times\Delta\uu,\uu\rangle\leqslant C\lVert \uu\rVert_{H^2}\lVert \uu\rVert_{L^2}.
\]
The coefficients $C=C(\lVert\tilde\m_{1}(t)\rVert_{H^{2}}, \lVert\tilde\m_{2}(t)\rVert_{H^{2}})$ in the estimates above are quadratic functions of~$\lVert\tilde\m_{1}(t)\rVert_{H^{2}}$ and $\lVert\tilde\m_{2}(t)\rVert_{H^{2}}$. Due to~\eqref{eq:propertyOfSolution2}, we can, for almost every~$\tilde\omega$, estimate these coefficients from above by a single constant $C=C(\tilde \omega)$. 

Combining the obtained estimates, we arrive at the inequality
\begin{equation}\label{eq:withR}
\frac{1}{2}\lVert\uu\rVert_{L^2}^2(t)+\int\limits_{0}^{t}\lambda\eps^2\lVert \uu\rVert_{\dot{H}^2}^2\,ds+\int\limits_{0}^{t}\lambda\lVert \uu\rVert_{\dot{H}^1}^2\,ds
\leqslant C \int\limits_{0}^{t}\lVert \uu\rVert_{L^2}^2+R(\lVert \uu\rVert_{L^2},\lVert \uu\rVert_{H^2})
\,ds,
\end{equation}
where the function $R(a, b)$ is given by 
\[R(a,b)=ab+a^{1/4}b^{7/4}+a^{1/2}b^{3/2}.
\]

We want to get rid of the homogeneous norm $\lVert \uu\rVert_{\dot{H}^{2}}$ on the right-hand side in~\eqref{eq:withR}. To this end,
we apply the following form of Young's inequality to each product of the $L^2$ and $H^2$-norms in the function $R(\lVert \uu\rVert_{L^2},\lVert \uu\rVert_{H^2})$: For any given $\delta\in(0,\infty)$ and $p\in(0,2)$, there exists a constant $C=C(\delta)$ such that for all $a,b\in(0,\infty)$, there holds
\[
a^{p}b^{2-p}\leqslant \delta a^{2}+C(\delta)b^{2}.
\]
In doing so, we always give the term $\lVert\uu\rVert_{H^2}^2$ an appropriately small prefactor, corresponding to the parameter $\delta$ above, so that we finally obtain the estimate
\[R(\lVert \uu\rVert_{L^2},\lVert \uu\rVert_{H^2})\leqslant \frac{\lambda\eps^2}{2}\lVert \uu\rVert_{H^2}^2+C\lVert \uu\rVert_{L^2}^2=\frac{\lambda\eps^2}{2}\lVert \uu\rVert_{\dot{H}^2}^2+(C+\frac{\lambda\eps^2}{2})\lVert \uu\rVert_{L^2}^2.
\] 
We absorb the homogeneous $H^{2}$-norm $\lVert\uu\rVert_{\dot{H}^2}^2$ on the left-hand side of~\eqref{eq:withR} and obtain that, for almost every $\tilde\omega$, there holds 
\[\frac{1}{2}\lVert\uu(t)\rVert_{L^2}^2+\frac{\lambda\eps^2}{2}\int\limits_{0}^{t}\lVert \uu(t)\rVert_{\dot{H}^2}^2\,ds+\lambda\int\limits_{0}^{t}\lVert \uu(t)\rVert_{\dot{H}^1}^2\,ds\leqslant  C(\tilde\omega)\int\limits_{0}^{t} \lVert\uu(t)\rVert_{L^2}^2\,ds.
\]
Since both integrals on the left-hand side above are non-negative, we conclude with the Gronwall lemma that $\lVert\uu(t)\rVert_{L^2}^2=0$ for all $t\in[0,T]$ and almost all $\tilde \omega$. Therefore, condition~\eqref{eq:UniqInL2} holds, which means that the weak martingale solution is pathwise unique. 
\end{proof}

We conclude this section with the proof of Corollary~\ref{Cor:Interpolation}.
\begin{proof}[Proof of Corollary~\ref{Cor:Interpolation}]
For the solution~$\m$ of~\eqref{SLLG+IC} and any $\beta\in(0,1/2)$ we have 
\begin{equation}\label{eq:containment}
\m\in L^{\infty}(0,T; H^{2})\cap C^{\beta}([0, T]; L^{2}) \text{ almost surely},
\end{equation}
due to claims (1) and (4) in Theorem~\ref{Thm:weak}. 

We now have to estimate the increment~$\lVert \m(t_{1})-\m(t_{2})\rVert_{C(\T^{3})}$ for $t_{1}, t_{2}\in[0,T]$. By the Sobolev embedding, we have that 
\[\lVert \m(t_{1})-\m(t_{2})\rVert_{C(\T^{3})}\leqslant C\lVert \m(t_{1})-\m(t_{2})\rVert_{H^{s}}
\]
for any $s\in(3/2, 2)$. Using~\eqref{eq:containment} and the fact (\cite{Taylor_I}, Proposition~4.3.1) that $H^{s}$ is obtained by an interpolation between $L^{2}$ and $H^{2}$, we see that 
\begin{multline*}
\lVert \m(t_{1})-\m(t_{2})\rVert_{H^{s}}\leqslant C\lVert \m(t_{1})-\m(t_{2})\rVert_{H^{2}}^{s/2}\lVert \m(t_{1})-\m(t_{2})\rVert_{L^{2}}^{1-s/2}\\
\leqslant C\lVert \m\rVert_{L^{\infty}(0,T; H^{2})}^{s/2}|t_{1}-t_{2}|^{\beta(1-s/2)}.
\end{multline*}
Since the exponent $\beta(1-s/2)$ belongs to the interval $(0,1/8)$, the proof is complete.
\end{proof}

\bibliographystyle{siam}

\end{document}